\documentclass[11pt,a4paper,oneside]{amsart}
\pdfoutput=1
\usepackage{amssymb,latexsym,graphics,color,vmargin}

\newtheorem{theorem}{Theorem}[section]

\newtheorem{remark}[theorem]{Remark}

\newtheorem{definition}{Definition}[section]
\theoremstyle{definition}

\newcommand{\beql}[1]{\begin{equation}\label{#1}}
\newcommand{\eeq}{\end{equation}}
\newcommand{\comment}[1]{}

\newcommand{\Abs}[1]{{\left|{#1}\right|}}

\newcommand{\Set}[1]{{\left\{{#1}\right\}}}

\newcommand{\RR}{{\mathbb R}}
\newcommand{\CC}{{\mathbb C}}
\newcommand{\ZZ}{{\mathbb Z}}
\newcommand{\NN}{{\mathbb N}}

\newcommand{\ft}[1]{\widehat{#1}}

\newcommand{\lcm}{{\rm lcm\,}}
\newcommand{\diam}{{\rm diam\,}}

\newcounter{rem}
\newcounter{step}
\newcommand{\step}[1]{\ \par \noindent\refstepcounter{step}{\bf \thestep.} #1}

\begin{document}

\title{Algorithms for translational tiling}

\author{Mihail N. Kolountzakis \& M\'at\'e Matolcsi}

\date{October 2008}

\address{M.K.: Department of Mathematics, University of Crete, Knossos Ave.,
GR-714 09, Iraklio, Greece} \email{kolount@gmail.com}

\address{M.M.: Alfr\'ed R\'enyi Institute of Mathematics,
Hungarian Academy of Sciences POB 127 H-1364 Budapest, Hungary.
\newline
 (also at BME Department of Analysis, Budapest,
H-1111, Egry J. u. 1)}\email{matomate@renyi.hu}

\thanks{MK: Supported by research grant No 2569 from the Univ.\ of Crete.}
\thanks{MM: Supported by Hungarian National Foundation for Scientific Research
(OTKA), Grants No. PF-64061, T-049301, T-047276}

\begin{abstract}
In this paper we study algorithms for tiling problems. We show
that the conditions $(T1)$ and $(T2)$ of Coven and Meyerowitz
\cite{coven-meyerowitz}, conjectured to be necessary and
sufficient for a finite set $A$ to tile the integers, can be
checked in time polynomial in $\mbox{diam}(A)$. We also give
heuristic algorithms to find all non-periodic tilings of a cyclic
group $\ZZ_N$. In particular we carry out a full classification of
all non-periodic tilings of $\ZZ_{144}$.

\end{abstract}

\maketitle



\section{Translational tiling}\label{sec:intro}

\begin{definition}
We say that a subset $A$ of an abelian group $G$ (written
additively)  {\em tiles $G$ by translation} when there exists a
subset $B \subset G$ such that every $g \in G$ has a unique
representation
$$
g = a + b,\ \ \ (a\in A, b\in B).
$$
This situation is denoted by $G = A\oplus B$ or $G = B\oplus A$.
The sets $A$ and $B$ will be called {\em tiling complements} of
each other.
\end{definition}

\begin{remark}
In much of the bibliography, especially the earliest one and that
with a more algebraic focus, the tiling condition $G = A \oplus B$
is often called a {\em factorization} of the group $G$.
\end{remark}

In this paper we are primarily interested in the case when $G$ is a finite
group, typically a cyclic group $\ZZ_N = \ZZ / (N\ZZ)$, or the group $\ZZ^d$, $d \ge 1$.

\subsection{Periodicity}\label{sec:periodicity}

\begin{definition}
A set $B \subseteq G$ is called {\em periodic} if there exists
a nonzero $g \in G$ such that $B + g = B$. Such a $g$ is then called
a {\em period} of $B$ and clearly the set of all periods plus $0$ forms
a subgroup of $G$.

In case $B\subseteq\ZZ$ is a periodic set of integers,
its least positive period is denoted by ${\mathcal P}(B)$.
\end{definition}

A periodic set $B$ is clearly a union of congruence classes $\bmod\ {\mathcal P}(B)$
and we can write
$$
B = \widetilde{B} \oplus {\mathcal P}(B)\ZZ,
$$
where $\widetilde{B} \subseteq [{\mathcal P}(B)\ZZ]$ and we have used the notation
$$
[n] = \Set{0,1,\ldots,n-1}.
$$
It has long been known \cite{newman}
that if $A$ is a finite set of integers and $\ZZ=A\oplus B$ is a tiling
of the integers then the set $B \subseteq \ZZ$ is necessarily periodic.
A consequence of this is that the study of any such tiling with period, say, $N={\mathcal P}(B)$,
reduces to the study  of a tiling of the group $\ZZ_N$, namely the tiling
$\ZZ_N = A\oplus \widetilde{B}$.

Non-periodic tilings of a group $G=A\oplus B$ are in some sense
less structured and more interesting. Indeed, if $h \in
G\setminus\Set{0}$ is a period of $B$ and $H \le G$ is the
subgroup generated by $h$, then we can write $B = B' \oplus H$ for
some $B' \subseteq G/H$ (slight, harmless abuse of language here).
The tiling condition $G = A \oplus B$ then becomes equivalent to
the tiling condition $G/H = A \oplus B'$. This means that
periodicity implies an immediate reduction to the complexity of
the problem. Haj\'os \cite{hajos} called a finite abelian group
$G$ a {\em good group} if in any tiling $G = A \oplus B$ one of
the two sets $A$ and $B$ must be periodic. Sands
\cite{sands-1,sands-2} completed the classification of all finite
abelian groups into good and bad groups (for the entire list see
also \cite[Prop.\ 4.1]{lagarias-wang}). The smallest cyclic group
that is not good is $\ZZ_{72}$.

\subsection{Tiling in Fourier space. Cyclotomic polynomials}\label{sec:tiling-in-fourier-space}

The condition $G=A \oplus B$ is clearly equivalent to
\beql{tiling}
\sum_{g \in G} \chi_A(g) \chi_B(x-g) = 1,\ \ (x \in G),
\eeq
which we can rewrite as $\chi_A * \chi_B = 1$, with $*$ denoting the convolution of two functions.

Assume for simplicity that $G$ is finite from now on.
For a function $f:G \to \CC$ its Fourier Transform \cite{rudin} is defined as the function $\ft{f}:\Gamma \to \CC$
given by
$$
\ft{f}(\gamma) = \sum_{g \in G} f(g) \overline{\gamma(g)},
$$
$\Gamma$ being the {\em dual group} of $G$. It consists of all characters $\gamma$ which are the group
homomorphisms $G \to \CC$.
It is easy to see that $\ft{f*g} = \ft{f}\cdot\ft{g}$.
Using this our tiling condition \eqref{tiling} can be rewritten as
$$
\ft{\chi_A} \cdot \ft{\chi_B} = \ft{1} = \Abs{G} \chi_{\Set{0}}.
$$
This is in turn equivalent to the obvious condition $\Abs{A}\cdot\Abs{B} = \Abs{G}$ as well as
\beql{zeros}
Z(\ft{\chi_A}) \cup Z(\ft{\chi_B}) = G \setminus \Set{0},
\eeq
where $Z(f) = \Set{x: f(x) = 0}$ is the zero set of $f$.
So, to check that two sets $A, B \subset G$ are tiling complements we have to first check the obvious
condition $\Abs{A}\cdot\Abs{B} = \Abs{G}$ and verify that the zero sets cover everything but $0$.

In the particular case where $G = \ZZ_N$ is a cyclic group, we have that $\Gamma$ is isomorphic to $G$
and the FT of $f:G \to \CC$ now becomes the function on $\ZZ_N$
$$
\ft{f}(k) = \sum_{j=0}^{N-1} f(j) e^{-2\pi i kj/N},\ \ (k = 0, 1, \ldots, N-1).
$$
Introducing the polynomial $P(X) = \sum_{j=0}^{N-1} f(j) X^j$ it is often convenient to
view the FT $\ft{f}$ as the evaluation of $P(X)$ at the $N$-th roots of unity
$$
1, e^{2\pi i \frac{1}{N}}, e^{-2\pi i \frac{2}{N}}, \ldots, e^{-2\pi i \frac{N-1}{N}}.
$$
It is the roots of this polynomial $P(X)$ that matter when one
checks condition \eqref{zeros}.  And in the case where $f =
\chi_A$ is the indicator function of a set $A \subseteq \ZZ_N$ the
polynomial $P(X)$ has integer coefficients. The irreducible (over
$\ZZ$) factors of $P(X)$ which are responsible for the zeros of
$P(X)$ at $N$-th roots of unity are {\em cyclotomic polynomials},
which are defined as the minimal polynomials of roots of unity
(see \cite{coven-meyerowitz} for a brief introduction to the
cyclotomic polynomials).

A quick way to define the cyclotomic polynomials $\Phi_n(X)$ is by the decomposition
of $X^N-1$ into irreducible factors
$$
X^N-1 = \prod_{d | N} \Phi_d(X),
$$
where the product extends over all divisors of $N$, $1$ and $N$ included.

The irreducible monic polynomial $\Phi_n(X)$ has as roots
precisely the primitive $n$-th roots of  unity, which are
therefore algebraic conjugates of each other
$$
\Phi_n(X) = \prod_{0 \le k <n,\ (k,n)=1} (X-e^{2\pi i k/n}).
$$
As a consequence, whenever an integer polynomial vanishes on a
primitive $n$-th root of unity it  vanishes on all of them. When
vanishing of such a polynomial is the issue, therefore, the set of
$N$-th roots of unity is split into $d(N)$ (= number of divisors
of $N$) blocks, and the zero set of such polynomial $P(X)$ among
the roots of unity is a union of such {\em cyclotomic blocks}.

\begin{remark}
Let us also observe that periodicity of a set $A \subseteq \ZZ_N$
can easily be detected on the Fourier side. Indeed, it is easy to
see that, if $N = ab$, the set $A \subseteq \ZZ_N$ is $a$-periodic
if and only if its Fourier transform $\ft{\chi_A}(k)$ vanishes on
all $k \in \Set{0,1,\ldots,N-1}$ which are not multiples of $b$.
\end{remark}

\subsection{The Coven-Meyerowitz conditions}\label{sec:cm-conditions}

Suppose $A$ is a finite set of nonnegative integers (we assume $0 \in A$) and write, as is customary,
$$
A(X) = \sum_{a \in A} X^a.
$$
Let $S_A$ be the set of prime powers $p^a$ such that $\Phi_{p^a}(X) \ |\  A(X)$.
In \cite{coven-meyerowitz} Coven and Meyerowitz wrote down the following two conditions on a such a polynomial $A(X)$.
\begin{itemize}
 \item[($T_1$)] $A(1) = \prod_{s \in S_A} \Phi_s(1)$,
 \item[($T_2$)] If $s_1, \ldots, s_m \in S_A$ are powers of distinct primes then
$\Phi_{s_1\cdots s_m}(X)\ |\ A(X)$.
\end{itemize}
They proved in \cite{coven-meyerowitz} that if ($T_1$) and ($T_2$)
hold for a set $A$ then $A$ tiles the integers by translation.
Another way to say this is that there is a number $N$ such that
$A$ tiles $\ZZ_N$. In the converse direction they proved that
($T_1$) necessarily holds for any tile $A$. Regarding the
necessity of ($T_2$) for tiling it was proved in
\cite{coven-meyerowitz}  that ($T_2$) is also necessary for tiling
when $\Abs{A}$ has at most two different prime factors. It was
conjectured by Konyagin and \L aba \cite{konyagin-laba} that $A$
is a tile of the integers if and only if both ($T_1$) and ($T_2$)
hold.

\subsection{The computational status of tiling}\label{sec:status}
The most basic  computational problem of tiling is to decide
whether a given finite set $A$ in an abelian group $G$ tiles the
group. If one ignores, as a first approach, questions of
complexity and restricts oneself to questions of decidability, one
must assume $G$ to be infinite (and, of course, discrete) for the
problem to be meaningful.

In a more general  form of the problem, that of asking whether a
given {\em set of tiles} can be moved around (by a group of
motions) to tile $\RR^d$, tiling has long been shown to be
undecidable. Berger \cite{berger} first showed this (it is
undecidable to determine if a given finite set of polygons can
tile $\RR^2$ using rigid motions). Many other models of tiling
have been shown to undecidable.

It is the case  of tiling by translations and by a single tile
that interests us here.

When the group  $G$ is the group of integers $\ZZ$ the problem is
decidable. This nontrivial fact follows from Newman's result
\cite{newman} that every translational tiling of $\ZZ$ by a finite
set is periodic. Although a bound for the period
($2^{\diam A}$) is given by Newman's theorem, this is not
strictly necessary to deduce the mere existence of an algorithm to
decide tiling by $A$ (if, that is, we do not care about the
running time of the algorithm). Indeed, given the finite $A
\subset \ZZ$, all we have to do is to start examining, by
exhaustive search, whether $A$ can tile the set
$\Set{-N,\ldots,N}$ for larger and larger $N$. Tiling here means
that we want to find a collection of non-overlapping translates of
$A$ which will cover $\Set{-N,\ldots,N}$. If the set $A$ cannot
tile $\ZZ$ then there is a finite $N$ for which we will not be
able to tile $\Set{-N,\ldots,N}$. This is a simple compactness (or
diagonalization) argument. On the other hand, if $A$ does tile
$\ZZ$ then, for large enough $N$, we will observe a tiling of
$\Set{-N,\ldots,N}$ which is ``periodic'' and which can,
therefore, be extended indefinitely to the left and right to tile
$\ZZ$.

The argument  using periodicity given above is very general (see
the introduction of \cite{robinson}, where it is also proved
that if a set admits a tiling of the plane with one period then
it also admits a fully periodic tiling), works
in all dimensions, and it is enough that periodic tilings exist.
It is not necessary for all tilings to be periodic (which fails to
be the case even in dimension 2 for as simple a tile as a
rectangle). Thus, the so-called {\em periodic tiling conjecture}
\cite{grunbaum-shepard,lagarias-wang-1d} ({\em everything that
tiles by translation can also tile periodically}) implies
decidability of translational tiling. This conjecture is still
open in all dimensions $d \ge 2$.

Already  when the group is $\ZZ^2$ the question of deciding if a
given finite set $A \subset \ZZ^2$ can tile by translation is wide
open, apart from the result of Szegedy \cite{szegedy} who gave an
algorithm for the special cases of $\Abs{A}$ being a prime or 4.
There are also algorithms for other special cases but these all
have topological conditions
\cite{wijshoff-van-leeuwen,girault-beauquier-nivat} on the tile
(e.g.\ to be simply connected).

Let  us now restrict ourselves to the decision problem of deciding
whether a given $A \subseteq \ZZ_N$ can tile $\ZZ_N$ by
translation. We are interested to study the computational
complexity of this problem regarding $N$ as the parameter. In
particular we'd like to have an algorithm which runs in time
$O(N^c)$ for some fixed $c>0$. Such an algorithm is still lacking
though except when some arithmetic conditions on $\Abs{A}$ are
imposed. In this paper we prove that if $\Abs{A}$ has at most two
prime factors then we can decide if $A$ is a tile of $\ZZ$ in
polynomial time. This is so because we can decide the
Coven-Meyerowitz conditions ($T_1$) and ($T_2$) (see
\S\ref{sec:cm-conditions}) in polynomial time (Theorem
\ref{th:cm-in-polynomial-time} below). We also introduce a local
version of the Coven-Meyerowitz conditions in \S\ref{sec:local}
which will allow us to give a similar algorithm, under the same
arithmetic conditions, for when $A$ tiles $\ZZ_N$.

On  the other hand, some very similar problems to the decision of
tiling (problem DIFF in \cite{kolountzakis-matolcsi}) have been
shown to be NP-complete. This would suggest a lower bound in the
computational complexity of the tiling decision problem. Were this
decision problem to prove to be NP complete this would refute the
equivalence of ($T_1$ \& $T_2$) to tiling (conjectured in \cite{konyagin-laba}),
assuming of course P$\neq$NP.

In the last section of the paper we turn to the problem of finding
all {\it non-periodic} tilings of a cyclic group $\ZZ_N$. As
explained in subsection \ref{sec:periodicity}, periodic tilings
are less interesting because they can be considered as tilings of
factor groups of $\ZZ_N$. Finding many (or, indeed all)
non-periodic tilings of $\ZZ_N$ could be a way of testing the
Coven-Meyerowitz conditions and possibly producing
counterexamples. Besides being mathematically interesting on its
own right, this problem has a particular motivation in certain
modern {\it music} compositions. The interaction of mathematical
theory and musical background has been extensively studied in
recent years \cite{aaa,amiot1,mor,fri1,vuza}.

H. Fripertinger listed out all non-periodic tilings of $\ZZ_{72}$
and $\ZZ_{108}$ with the help of a computer search \cite{fri}. In
this paper we restrict our attention to $N=144$, but we believe
that the methods described here can be used to classify all
non-periodic tilings for all $N\le 200$. We also see from Section
\ref{sec:number-of-complements} that the number of non-periodic
tilings grows at least exponentially with $N$, so that for large
and highly-composite values of $N$ this task gets hopeless. Also,
the motivation for musical compositions does not extend beyond
$N=200-300$ due to obvious perceptional limitations. Let us also
mention that in order to test the Coven-Meyerowitz conditions one
will need to list non-periodic tilings of $\ZZ_N$ for values of
$N$ containing at least 3 different prime factors, such as $N=120,
180, 200$, or even much higher.

\section{Deciding the Coven-Meyerowitz conditions}\label{sec:deciding-cm-conditions}

In this section we describe in detail an algorithm which, given a set of integers
$$
A \subseteq \Set{0,\ldots,D}
$$
decides if $A$ satisfies the ($T_1$) and ($T_2$) conditions of Coven and Meyerowitz \cite{coven-meyerowitz}.
\begin{theorem}\label{th:cm-in-polynomial-time}
There is an algorithm to decide whether conditions ($T_1$)  and
($T_2$) (see \S\ref{sec:cm-conditions}) hold for a given $A
\subseteq \ZZ$ which runs in time polynomial in $D=\mbox{diam
}(A)$. This algorithm therefore decides if $A$ tiles $\ZZ$ when
$\Abs{A}$ contains at most two distinct prime factors.
\end{theorem}
\begin{proof}
The algorithms consists of the steps given below.

\step{
\underline{Compute all cyclotomic polynomials of degree up to $D$.}

This step is very easy to carry out in time polynomial in $D$ using, for instance,
the formula
$$
\Phi_n(x) = \prod_{d|n} (1-x^{n/d})^{\mu(d)},
$$
where $\mu(d)$ is the M\"obius function.
}

\step{
\underline{Determine the cyclotomic divisors of the polynomial $A(x)$.}

Again, this is easily doable in time polynomial in $D$. Let $S_A$
be the set of indices $n=p_i^{a_i}$, which are prime powers, such
that $\Phi_n(x) | A(x)$. Let $p_1, p_2, \dots p_k$ be the
different primes whose powers appear here, and let $N_i \ge 1$,
$i=1,\ldots,k$, be the number of relevant powers of the prime
$p_i$. }

\step{
\underline{Test if condition ($T_1$) holds.}

Having computed the set $S_A$ in the previous step, this amounts to checking the definition
of condition ($T_1$).
}

\step{
\underline{If $(N_1+1)(N_2+1)\cdots(N_k+1) - 1> D$ then answer that ($T_2$) fails. End.}

The explanation of this step is as follows. If ($T_2$) is to hold then for any choice of prime powers
(with respect to different primes) from the set $S_A$ there is a different divisor of $A(x)$,
namely the cyclotomic polynomial $\Phi_n(x)$ where $n$ is the product of the chosen prime powers.
Comparing degrees we obtain the inequality
$$
(N_1+1)(N_2+1)\cdots(N_k+1) - 1 \le D.
$$
}

\step{
\underline{Check exhaustively if ($T_2$) holds and reply accordingly.}

This exhaustive search has to check $(N_1+1)(N_2+1)\cdots(N_k+1) - 1$ possibilities
for ($T_2$) to fail (select a power of each involved prime $p_i$, with exponent from $0$ to $N_i$).
This number is at most $D$ (by the previous step) and therefore the total cost of this step of the algorithm
is polynomial in $D$.
}

The  fact that this algorithm decides if $A$ tiles $\ZZ$ when
$\Abs{A}$ contains at most two distinct prime factors follows from
Theorem B2 in \cite{coven-meyerowitz} which claims that, in that
case, the conjuction $T_1$ and $T_2$ is equivalent to tiling
$\ZZ$.
\end{proof}

\section{The local Coven-Meyerowitz conditions}\label{sec:local}
In this  section we give a local version of the Coven-Meyerowitz
conditions, relevant to tiling a given cyclic group $\ZZ_N$ and
not $\ZZ$, and derive the corresponding properties they enjoy.

Let $N$ be a positive integer.
Define $S_A^N$ be the set of prime powers $p^a | N$ such that $\Phi_{p^a}(X) \ |\  A(X)$.
Similarly define the conditions
\begin{itemize}
 \item[($T_1^N$):] $A(1) = \prod_{s \in S_A^N} \Phi_s(1)$,
 \item[($T_2^N$):] If $s_1, \ldots, s_m \in S_A^N$ are powers of distinct primes then
$\Phi_{s_1\cdots s_m}(X)\ |\ A(X)$.
\end{itemize}
In Theorem  \ref{th:local} below statements
(\ref{lb:local-imply-tiling}) and
(\ref{lb:tiling-implies-local-t1}) correspond to Theorems A and B1
of \cite{coven-meyerowitz}. Statement
(\ref{lb:conjecture-implies-local-conjecture}) means that the
Coven-Meyerowitz conjecture is stronger than the corresponding
local conjecture. Statement (\ref{lb:two-primes-local}),
corresponding to Theorem B2 in \cite{coven-meyerowitz}, means that
we know the local condition $T_2^N$ is also true when at most two
primes are involved in $\Abs{A}$. Finally, statement
(\ref{lb:algorithm-local}) means that we have a polynomial time
(in $N$) algorithm to decide the local Coven-Meyerowitz conditions
$T_1^N$ and $T_2^N$.
\begin{theorem}\label{th:local}
For any finite set $A \subseteq \ZZ$ and any positive integer $N$:
\begin{enumerate}
\item \label{lb:local-imply-tiling}
$T_1^N$ and $T_2^N$ $\Longrightarrow$ $A$ tiles $\ZZ_N$
\item \label{lb:tiling-implies-local-t1}
$A$ tiles $\ZZ_N$ $\Longrightarrow$ $T_1^N$
\item \label{lb:conjecture-implies-local-conjecture}
If ($A$ tiles $\ZZ$ $\Longrightarrow$ $T_1$ and $T_2$) then
($A$ tiles $\ZZ_N$ $\Longrightarrow$ $T_1^N$ and $T_2^N$).
\item \label{lb:two-primes-local}
If there are at most two distinct primes in $\Abs{A}$ and $A$ tiles $\ZZ_N$ then $T_1^N$ and $T_2^N$ hold.
\item \label{lb:algorithm-local}
There is an algorithm which decides if $T_1^N$ and $T_2^N$ hold in
time polynomial in $N$. This algorithm therefore decides if $A$
tiles $\ZZ_N$ when $\Abs{A}$ contains at most two distinct prime
factors.
\end{enumerate}
\end{theorem}
\begin{proof}
The proof of (\ref{lb:local-imply-tiling}) is  essentially the
same as that of Theorem A in \cite{coven-meyerowitz}. When reading
that proof it might help to observe that $A(1)B(1) = \lcm(S_A)$,
or, in the local version, $A(1)B(1) = \lcm(S_A^N)$.

For the proof of (\ref{lb:tiling-implies-local-t1})   we use Lemma
2.1 in \cite{coven-meyerowitz}. According to that, if $A$ tiles
$\ZZ_N$ it follows that $S_A$ contains only divisors of $N$. But
$A$ tiles $\ZZ_N$ implies that $A$ tiles $\ZZ$, hence, by Theorem
B1 of \cite{coven-meyerowitz}, $T_1$ is valid. But $S_A = S_A^N$
in this case hence $T_1^N$ is valid too.

To prove (\ref{lb:conjecture-implies-local-conjecture}),  that the
Coven-Meyerowitz conjecture  implies the corresponding local
conjecture one first observes that if $A$ tiles $\ZZ_N$ then it
also tiles $\ZZ$. By the Coven-Meyerowitz conjecture follows the
validity of $T_1$ and $T_2$ and, finally, observe that $T_2$
implies $T_2^N$ for all $N$.

To show  (\ref{lb:two-primes-local}) we have, if $A$ tiles
$\ZZ_N$, that $A$ tiles $\ZZ$, hence $T_2$ holds (by Theorem B2 in
\cite{coven-meyerowitz}), which implies that $T_2^N$ holds.

To prove  (\ref{lb:algorithm-local}) we merely repeat the
algorithm of \S \ref{sec:deciding-cm-conditions} but replacing
$S_A$ by $S_A^N$, etc. Because of (\ref{lb:two-primes-local}) this
algorithm also decides if $A$ tiles $\ZZ_N$ if $\Abs{A}$ contains
at most two distinct prime factors.
\end{proof}

\section{The number of non-periodic tiling complements} \label{sec:number-of-complements}

One might expect that the number of non-periodic tilings of $\ZZ_N$
is small. However this is not the case.
\begin{theorem}\label{exp}
There are arbitrarily large $N$ and non-periodic tilings $\ZZ_N =
A \oplus B$, such that  there are additional distinct non-periodic
tiling complements $B_1,\ldots,B_k$ of $A$, with $k \ge
e^{C\sqrt{N}}$, $C$ a constant.
\end{theorem}

\begin{proof}
The following sketch of  a proof relies on a construction given in
\cite{long-periods} where it was used to prove that there are
tilings of the integers with a tile contained in
$\Set{0,\ldots,D}$ but of period at least $C D^2$.

We take $N=2\cdot 3 \cdot 5 \cdot p \cdot q$ where $p$  and $q$
are two different large primes, roughly of the same size $\sim
\sqrt N$. We view the group $\ZZ_N$ as
$$
\ZZ_N = \ZZ_{3p} \times \ZZ_{5q} \times \ZZ_2
$$
as shown in Fig.\ \ref{fig:layers}.
\begin{figure}
 \begin{center} \resizebox{10cm}{!}{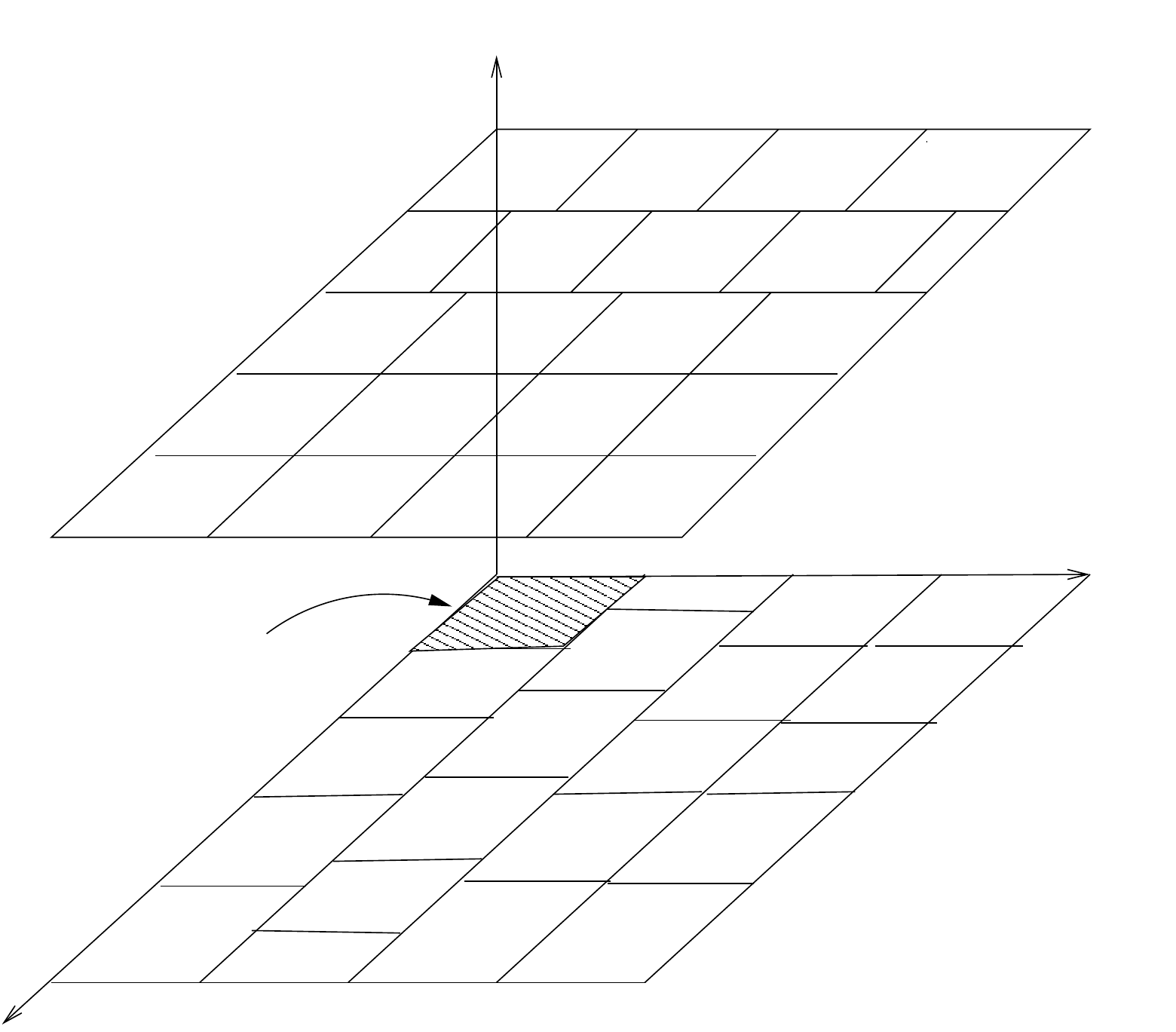} \end{center}
\caption{The group $\ZZ_{3p} \times \ZZ_{5q} \times \ZZ_2$}
\label{fig:layers}
\end{figure}
We take the set $A$ to be the ``$3\times 5$ rectangle at the
origin''. The set $A$ can then tile $\ZZ_N$ in a non-periodic way
by first tiling the lower and the upper layer in the ordinary way
and then perturbing a row of the lower layer and a column of the
upper layer, as shown in Fig.\ \ref{fig:layers}. Each such
perturbation manages to destroy one of the two periods that
existed in the ordinary tiling (for a detailed proof see
\cite{long-periods}).

If one wants to create many different non-periodic tilings of $\ZZ_N$ with $A$ one does as follows.
\begin{enumerate}
\item
Keep the upper layer the same (with one perturbed row only).
\item
Keep half  the lower layer the same: rows indexed from $5q/2$ to
$5q$ are left unperturbed (refer to Fig.\ \ref{fig:layers}).
\item
Perturb  the remaining rows of the lower layer arbitrarily,
subject only to the restriction that at least one of them is
perturbed by a non-zero amount.
\end{enumerate}
It is clear that the number of such complements is at least
$$
\sim 3^{5q/2} \sim e^{C\sqrt{N}}.
$$
Each such  tiling is non-periodic: suppose $g=(a,b,c) \in \ZZ_{3p}
\times \ZZ_{5q} \times \ZZ_2$ is a period. The two layers cannot
be swapped by a translation by $g$, as one has perturbed rows but
the other one does not, so $c=0$. The upper layer is mapped to
itself by a $g$-translation. This can only happen if $b=0$ as the
perturbed column is unique. Finally $a=0$ since the lower layer
cannot move in the direction of $\ZZ_{5q}$ and be mapped into
itself: the perturbed row closest (in the ``direction of motion'')
to the unperturbed block of $\sim 5q/2$ rows cannot move to a
perturbed row.

Even if one  considers two tilings which differ by a translation
identical the number would still remain exponential in $N$.
\end{proof}

\section{Finding all non-periodic tilings in a cyclic group}\label{sec:all-tilings}

As explained in Section \ref{sec:intro} there is a particular
motivation coming from music compositions to list out {\it all
non-periodic} tilings of a cyclic group $\ZZ_N$, for relatively
small values of $N$. H. Fripertinger \cite{fri} achieved this task
for $N=72, 108$. In this section we settle $N=144$, and we believe
that the algorithms described here are likely to work for other
values like $N=120, 180, 200, 216$. On the other hand we see from
Theorem \ref{exp} that this task becomes hopeless for large and
``highly composite'' $N$.

We cannot offer any rigorous mathematical statements here with
regard to the complexity of the algorithms used. We provide
heuristics, observe that methods seem to work remarkably well and
we are able to achieve a full classification for $N=144$.

\begin{remark}
Throughout this section we rely on results of
\cite{coven-meyerowitz}. Our algorithm works only when tiling
implies the $T_2^N$ condition (see \S\ref{sec:local}). In
particular it works when either $N$ has at most two distinct prime
factors, such as $N = 144 = 2^4\cdot 3^2$, or has three distinct
prime factors but one of them appears to the first power, such as
$N = 120 = 2^3\cdot3\cdot 5$ or $N = 180 = 2^2\cdot 3^2 \cdot 5$.
The reason for this is that in any tiling $\ZZ_N = A \oplus B$ in
such a case either $\Abs{A}$ or $\Abs{B}$ has at most two distinct
primes factors.
\end{remark}

A basic ingredient of our algorithm is the ``fill-out procedure''
performing the following task: given a set $A$ in a group $G$,
list out {\it all} tiling complements $B$ of $A$. We cannot give
any precise statement on how efficiently this task can be
achieved except to note that it performs really well in practice.
However, let us mention here that the method described
below was already successfully used in \cite{mora} in refuting a
conjecture of Lagarias and Szab\'o on a necessary condition for
the existence of universal spectrum, and it could well be useful
in tackling other question related to tiling. The credit for this
heuristic algorithm goes to P. M\'ora in \cite{mora}.

\subsubsection*{The Fill-out Procedure}

\ \\

Assume $A\subset \ZZ_N$ is given (in fact, the procedure works for
any finite abelian group $G$, but we restrict our attention to
cyclic groups). The task is to find all tiling complements $B$ of
$A$. As usual, we make the normalization assumption that $0\in A,
B$.

We will build up, by adding elements one by one, all tiling
complements $B$ of $A$. For this, we move in the space of {\em
packing complements} $P$ of $A$, i.e. sets $P=\{p_1, \dots p_m\}$
such that all translated copies $A+p_j$ are disjoint from each
other. We are trying to grow the sets $P$ so that they become {\it
tiling complements}. Our exploration starts with the set $P =
\Set{0}$. At any given step we will attempt to extend $P$ by one
new element in a particular manner (which is described in detail
in the next paragraph). If no such element exists we backtrack. We
describe our algorithm ``{\tt explore}'' below, and note that {\tt
explore}($P$) finds all tiling complements $B$ containing $P$. We
invoke it as {\tt explore}($\Set{0}$) with the list of tiling
complements of $A$ being initially empty. The recursive procedure
{\tt explore} is described in Fig.\ \ref{fig:explore}.

\begin{figure}[ht]

{\tt explore}($P$)\hfil
\begin{enumerate}
\item
If $P$ has already been explored, {\tt return}.
\item
Mark $P$ as explored.
\item
If $P$ is a tiling complement of $A$, add it to the list of tiling
complements and {\tt return}.
\item
Compute the function $r:\ZZ_N \setminus (A\oplus P) \to \NN$,
defined by $r(x)=$ number of ways to add an extra copy of $A$ to
$A\oplus P$ so that $x$ is now covered and the new copy of $A$
does not intersect $A \oplus P$.
\item
Rank all $x \in \ZZ_N \setminus (A\oplus P)$ according to $r(x)$:
$$
r(x_1) \le r(x_2) \le \cdots \le r(x_k),
$$ in such a way that if $r(x_i)=r(x_{i+1})$ then $x_i<x_{i+1}$
(for this purpose of ordering the elements of $\ZZ_N$ are thought
of as $\{0,1,\dots N-1\}$). Here $k = \Abs{\ZZ_N \setminus
(A\oplus P)}.$
\item
If $r(x_1) = 0$, {\tt return}.
\item
Consider $x_1$.
\begin{itemize}
\item
Let $A+y_1, \ldots, A+y_{r(x_1)}$ be the copies of $A$ that can be
added to $A \oplus P$ in a non-overlapping way and that contain
$x_1$.
\item
{\tt explore} ($P \cup \Set{y_1}$), $\ldots$, {\tt explore} ($P
\cup \Set{y_{x(r_1)}}$).
\end{itemize}
\end{enumerate}

\caption{The {\tt explore} recursive procedure}
\label{fig:explore}
\end{figure}

The efficiency of this procedure is due to the heuristic of
expanding a given packing complement $P$ of $A$ by adding to it a
copy of $A$ which covers an element $x$ that is the most
restrictive in the sense that there are few possible ways to cover
it. This is achieved by ranking the yet uncovered elements $x$ in
increasing order of the function $r(x) = r(P,x)$ that counts in
how many ways an element $x$ can be covered by adding one more
copy of $A$ to the given packing $P$. Clearly, if $r(x)=0$ for
some $x$ then $x$ cannot be covered in any admissible way, so that
$P$ can definitely not be extended to a tiling complement (it is
still possible that $P$ could be extended by several further
elements to form a larger packing complement $P'$, but it can
never grow to become a {\it tiling} complement due to $x$ never
being covered). It also often happens that $r(x)=1$, i.e. we are
{\it forced} to place the next copy of $A$ at a certain position
in order to cover $x$, and once this copy is placed, there is even
less room in $\ZZ_N$, and it is likely that there will be another
$x$ such that $r(x)=1$, or at least small.

With this fill-out procedure at hand we can now describe the
algorithm to list all non-periodic tilings of $\ZZ_{144}$. We
merely describe the algorithm here, as the full documentation (or
listing out the numerous arising sets) would be far too lengthy
for this paper. The results and documentation are available online
at \cite{web}.

\vskip 0.5 truecm

{\bf Normalizing conditions.} We always assume that $0\in A,B$. Also, it
is clear that in terms of tiling questions a set $A=\{0, a_1, a_2,
\dots \}$ is equivalent to its own translated copies $A-a_1$,
$A-a_2$, $\dots$ all containing zero. It is customary to include
only one representative from these equivalent copies, and we are
going to follow this tradition (so that the numbers described in
the last section correspond to this normalization).

\subsubsection*{The Algorithm}

\ \\

{\it Step 1.} The prime powers dividing 144 are 2, 4, 8, 16, 3 and
9. In any tiling $A\oplus B=\ZZ_{144}$ the cyclotomic polynomials
corresponding to these prime powers must divide \underline{exactly} one of
$A(x)$ and $B(x)$, according to condition $T_1$ of
\cite{coven-meyerowitz}. Therefore we first make a list of all
possible partitions $\{H, H^c\}$ of the elements
$\{2,4,8,16,3,9\}$. There are 32 such partitions (note that $\{H,
H^c\}$ and $\{H^c, H\}$ are the same). Our task is to decide for
each partition whether any non-periodic tilings correspond to it.

\vskip 0.5 truecm

{\it Step 2.} Certain partitions produce only periodic tilings due
to condition $T_2$ of \cite{coven-meyerowitz}. Indeed, as an
example, take the partition $\{\{2,4\},\{8,16,3,9\}\}$. Due to
condition $T_2$ of \cite{coven-meyerowitz} the cyclotomic
polynomials $\Phi_{24}(x),\Phi_{72}(x),\Phi_{48}(x),\Phi_{144}(x)$
must also divide $B(x)$, together with
$\Phi_{8}(x),\Phi_{16}(x),\Phi_{3}(x),\Phi_{9}(x)$. From this it
follows that the support of $\ft{\chi_B}$ is contained in the
subgroup $\Set{0,2,4,\ldots,142}$ and this automatically makes $B$
periodic. Therefore, in this step, discard all partitions which
imply automatic periodicity of either $A$ or $B$ by condition
$T_2$.

The remaining partitions are:

\noindent $\{\{3, 4, 8\}$, $\{2, 9,16\}\}$, $\{\{3, 4, 9\}$, $\{2,
8, 16\}\}$, $\{\{3, 4, 16\}$, $\{2, 8, 9\}\}$, $\{\{3, 8, 9\}$,
$\{2, 4, 16\}\}$, $\{\{3, 8, 16\}$, $\{2, 4, 9\}\}$, $\{\{4, 8,
9\}$, $\{2, 3, 16\}\}$, $\{\{4, 8, 16\}$, $\{2, 3, 9\}\}$, $\{\{4,
9, 16\}$, $\{2, 3, 8\}\}$, $\{\{8, 9, 16\}$, $\{2, 3, 4\}\}$,
$\{\{2, 3, 4, 8\}$, $\{9, 16\}\}$, $\{\{2, 3, 4, 16\}$, $\{8,
9\}\}$, $\{\{2, 3, 8, 9\}$, $\{4, 16\}\}$, $\{\{2, 3, 8, 16\}$,
$\{4, 9\}\}$, $\{\{2, 4, 8, 16\}$, $\{3, 9\}\}$, $\{\{2, 4, 9,
16\}$, $\{3, 8\}\}$, $\{\{2, 8, 9, 16\}$, $\{3, 4\}\}$, $\{\{3, 4,
8, 16\}$, $\{2, 9\}\}$, $\{\{4, 8, 9, 16\}$, $\{2, 3\}\}$, $\{\{2,
3, 4, 8, 16\}$, $\{9\}\}$, $\{\{3,4, 8, 9\}$, $\{2, 16\}\}$,
$\{\{2, 3, 4, 9\}$, $\{8, 16\}\}$, $\{\{2, 3, 4, 8, 9\}$,
$\{16\}\}$.

\vskip 0.3 truecm

We remark here that most of these partitions will {\it not}
produce non-periodic tilings. If there were further theoretical
considerations upon which certain partitions could be discarded,
then it would probably be possible to apply our algorithm for
higher values of $N$. For instance, it could be true that if
either $|A|$ or $|B|$ is a prime power then the tiling must be
periodic. This would allow us to discard several further
partitions.

\vskip 0.3 truecm

Assume we are dealing with a fixed partition $P=\{H, H^c\}$. The
natural approach is the following: by the structure theory given
in Section 4 of \cite{coven-meyerowitz} one can list out all
subsets $A\subset \ZZ_{144}$ such that $A$ tiles $\ZZ_{144}$ and
$\Phi_h (x)$ divides $A(x)$ for all $h\in H$.

We remark here that besides using the structure theory of
\cite{coven-meyerowitz} there is an alternative way of listing out
these sets $A$. Namely, take the least common multiple $L$ of the
prime powers appearing in $H$. By the Remark following the proof
of Theorem A in \cite{coven-meyerowitz} one can construct a
universal tiling complement $B$ in $\ZZ_L$ of any tile $C$ in $\ZZ_L$
such that $\Phi_h (x)$ divides $C(x)$ for all $h\in H$. The set
$B$ is constructed by taking the product of certain cyclotomic
polynomials as described in the proof of Theorem A of
\cite{coven-meyerowitz}. Next we apply our fill-out procedure to
$B$ in $\ZZ_L$, and hence list out all possible sets $C$ in $\ZZ_L$.
Finally, all the desired tiles $A$ in $\ZZ_N$ must reduce to one of
these sets $C$ modulo $L$. Therefore, to list out the sets $A$ we
must shift the elements of sets $C$ by multiples of $L$ in all
possible ways.

Then these sets $A$ need to be grouped into equivalence classes
according to the zero set of the Fourier transform of $\chi_A$. A
similar procedure must be done for subsets $B$. Then one needs to
select the equivalence classes which correspond to non-periodic
sets $A$ and $B$ (this can be read off from the zero-sets of the
Fourier transform, or directly from any particular representative
of the class). Finally, if a class of non-periodic $A$'s and a
class of non-periodic $B$'s are such that the union of the
zero-sets of their Fourier transform contains $\ZZ_{144}\setminus
\{0\}$, then any representatives of these classes tile with each
other in a non-periodic way.

For some partitions this approach can indeed be carried out.
However, the problem is that for certain partitions the number of
tiles $A$ and/or $B$ provided by the structure theory of
\cite{coven-meyerowitz} can be overwhelmingly large. Take the
example $\{\{2, 3, 8, 16\}, \{4, 9\}\}$. The structure theory of
\cite{coven-meyerowitz} provides two different tiles $A_1, A_2$ of
cardinality 24 modulo $\lcm (2,3,8,16)=48$, such that $A(x)$ is
divisible by each of $\Phi_\Set{2,3,8,16}(x)$. The tiles $A$ modulo
144 are therefore obtained as sets that reduce to $A_1$ or $A_2$
modulo 48. There are $2\cdot 3^{23}$ such sets.

\vskip 0.5 truecm

{\it Step 3.} To overcome this problem, we only list one type of
the tiles (either $A$'s or $B$'s, whichever are fewer). In the
above example we list the tiles of type $B$. The structure theory
of \cite{coven-meyerowitz} provides six different tiles $B_1,
\dots B_6$ of cardinality 6 modulo $\lcm (4,9)=36$, such that
$B(x)$ is divisible by each of $\Phi_\Set{4,9}(x)$. The tiles $B$
modulo 144 are therefore obtained as sets that reduce to any of
$B_1, \dots B_6$ modulo 36. There are $6\cdot 4^{5}=6144$ such
sets. We can easily sort them into equivalence classes according
to the zero-set of their Fourier transforms. There are five
classes $C_1,\ldots, C_5$ corresponding to non-periodic tiles $B$.
They are characterized by the cyclotomic divisors $\Phi_h(x)$,
with $h$ being in $C_1=\{72, 36, 18, 9, 4\}$, $C_2=\{36, 12, 9,
4\}$, $C_3=\{36, 18, 9, 4\}$, $C_4=\{36, 18, 12, 9, 4\}$,
$C_5=\{72, 36, 18, 12, 9, 4\}$.

\vskip 0.5 truecm

{\it Step 4.} Next we discard those equivalence classes $C_j$
which automatically make any tiling complement $A$ periodic.
Consider, for example, $C_1$ above so that our tile $B$ has
cyclotomic divisors $\Phi_{\{72, 36, 18, 9, 4\}}$. Then any tiling
complement $A$ of $B$ must have cyclotomic divisors $\Phi_{\{2, 3,
6, 8, 12, 16, 24, 48, 144\}}$ (and possibly others), which
automatically make $A$ periodic.

\vskip 0.5 truecm

{\it Step 5.} For any remaining equivalence class $C_j$ we take a
representative $B_j$, and run our "fill-out procedure" to find
{\it all} tiling complements $A$ of $B_j$. The point is that $B_j$
is non-periodic, so that heuristically we expect not too many
tiling complements of $B_j$ to exist (although this heuristic
breaks down if $N$ is large and highly composite, as shown by
Theorem \ref{exp}). Once we have all tiling complements $A$ of
$B_j$ we can select the non-periodic ones, if any. Let $N_j$
denote the collection of all non-periodic tiling complements $A$
of $B_j$. Then we conclude that all sets $A$ in $N_j$ tile with
all representatives $B$ of the class $C_j$. We repeat this
procedure for all $j$ and we arrive at a list of {\it all}
non-periodic tilings corresponding to the fixed partition $P$. In
the specific example above, Step 4 already discards all classes
$C_j$ so that Step 5 becomes unnecessary.

We repeat Steps 3-5 for all partitions $P$ listed in Step 2, and
we arrive at a list of {\it all} non-periodic tilings of
$\ZZ_{144}$.

As a last normalization step, we must recall that a tile
$A=\{0,a_1, \dots \}$ and its shifted copies $A-a_1$, etc. are
considered equivalent, and keep only one representative for each
tile. And, naturally, the same applies to the tiles of type $B$.

\vskip 0.5 truecm

{\it Step 6.} {\bf Exceptional cases}. We have been able to execute the
steps above for all partitions listed in Step 2 except for $\{\{2,
4, 8, 16\}, \{3, 9\}\}$. In that case the number of tiles provided
by the structure theory of \cite{coven-meyerowitz} is too large
for both $A$ and $B$, and therefore the numerical search would
take several days. We choose to deal with this case by invoking
some results from the literature. We will show that either $A$ or
$B$ must be periodic.

Note that $|A|=16$ and $|B|=9$ by condition $T_1$ of
\cite{coven-meyerowitz}. Recall the following.
\begin{theorem}\label{th:sands} {\rm (Sands \cite{sands})}
If $A\oplus B=\ZZ_N$ and $N$ has at most two prime factors
$p,q$ then either $A$ or $B$ is contained in the subgroup $p\ZZ_N$
or $q\ZZ_N$.
\end{theorem}

Assume $|A|$ is contained in a subgroup. Then the subgroup must be
$H_3=\{0,3,6,\dots 144\}$ due to obvious cardinality reasons. Then
$B$ is a union of three parts $B_0$, $B_1$, $B_2$ according to
residues modulo 3. All three parts must have cardinality 3, and
$A\oplus B_0$ gives a tiling of the subgroup $H_3$. Note that $H_3$ is
isomorphic to $\ZZ_{48}$. If $A$ is not periodic then $B_0$ must be
periodic (all tilings of $\ZZ_{48}$ are periodic). Therefore we must
have $B_0=\{0,48,96\}$. Similarly, $B_1$ and $B_2$ must be
periodic by 48, and therefore so is $B$.

Assume now that $B$ is contained in a subgroup. Then the subgroup
must be $H_2=\{0,2,4,\dots, 144\}$. Then $A$ is a union of $A_0$
and $A_1$, the even and odd parts of $A$, and $A_0 \oplus B=H_2$. Note
that $H_2$ is isomorphic to $\ZZ_{72}$. If $B$ is non-periodic then
$A_0$ must be periodic (the only non-periodic tilings of $\ZZ_{72}$
contain tiles of cardinality 6 and 12, while $|B|=9$ and
$|A_0|=8$). In fact, $A_0$ must be periodic by 72, because the
periodicity subgroup contained in $A_0$ must divide the
cardinality $|A_0|=8$. By the same reasoning we obtain that $A_1$
is also periodic by 72, and therefore so is $A$.

\subsubsection*{Summary of  the results for $\ZZ_{144}$.}

\ \\

For the partition $\{\{2, 8, 9, 16\}, \{3, 4\}\}$ there is a set
$T_1$ containing 36 non-periodic tiles of cardinality 24 (type
$A$), and a set $T_2$ containing 6 non-periodic tiles of
cardinality 6 (type $B$). Each tile in $T_1$ tiles with each in
$T_2$. The 36 tiles in $T_1$ fall into two different equivalence
classes corresponding to the following cyclotomic factors:
$\Phi_{\{144, 72, 24, 18, 16, 9, 8, 2\}}$,and $\Phi_{\{144, 72,
18, 16, 9, 8, 2\}}$. The 6 tiles in $T_2$ fall into one
equivalence class corresponding to the cyclotomic factors
$\Phi_{\{48, 36, 24, 12, 6, 4, 3\}}$. We present here one
representative for $T_1$ and $T_2$ (one can then easily recover
the full sets by applying the "fill-out procedure" to these
representatives):
\begin{itemize}
\item
$\{$0, 17, 20, 23, 28, 29, 40, 48, 53, 59, 65, 68, 76,
88, 89, 95, 96, 101, 116, 124, 125, 131, 136, 137$\}$ is in
$T_1$ and
\item
$\{0, 32, 58, 90, 112, 122\}$ is in $T_2$.
\end{itemize}

\vskip 0.5 truecm

For the partition $\{\{4, 9, 16\}, \{2, 3, 8\}\}$ there is a set
$T_1$ containing 6 non-periodic tiles of cardinality 12 (type
$A$), and a set $T_2$ containing 324 non-periodic tiles of
cardinality 12 (type $B$). Each tile in $T_1$ tiles with each in
$T_2$. The 6 tiles in $T_1$ fall into one equivalence class
corresponding to the following cyclotomic factors: $\Phi_{\{144,
36, 18, 16, 9, 4\}}$. The 324 tiles in $T_2$ fall into two
equivalence classes corresponding to the cyclotomic factors
$\Phi_{\{72, 48, 24, 18, 12, 8, 6, 3, 2\}}$, and $\Phi_{\{72, 48,
24, 12, 8, 6, 3, 2\}}$. We present here one representative for
$T_1$ and $T_2$ (one can then easily recover the full sets by
applying the "fill-out procedure" to these representatives):
\begin{itemize}
\item
$\{0, 34, 40, 46, 48, 58, 88, 96, 106, 118, 130, 136\}$
is in $T_1$ and
\item
$\{0, 16, 29, 44, 57, 73, 80, 93, 108, 109,
124, 137\}$ is in $T_2$.
\end{itemize}

\vskip 0.5 truecm

For the partition $\{\{2, 4, 9, 16\}, \{3, 8\}\}$ there is a set
$T_1$ containing 8640 non-periodic tiles of cardinality 24 (type
$A$), and a set $T_2$ containing 3 non-periodic tiles of
cardinality 6 (type $B$). Each tile in $T_1$ tiles with each in
$T_2$. The 8640 tiles in $T_1$ fall into three different
equivalence classes corresponding to the following cyclotomic
factors: $\Phi_{\{144, 36, 18, 16, 12, 9, 4, 2\}}$, $\Phi_{\{144,
36, 18, 16, 9, 6, 4, 2\}}$, and $\Phi_{\{144,
  36, 18, 16, 9, 4, 2\}}$. The 3 tiles in $T_2$ fall into one equivalence
class corresponding to the cyclotomic factors $\Phi_{\{72, 48, 24,
12, 8, 6, 3\}}$ We present here one representative for $T_1$ and
$T_2$ (one can then easily recover the full sets by applying the
"fill-out procedure" to these representatives):
\begin{itemize}
\item
$\{$ 0, 9, 17, 26, 27, 34, 39, 40, 48, 51, 57, 65, 74, 82,
88, 96, 99, 105, 111, 113, 122, 123, 130, 136$\}$ is in
$T_1$ and
\item
$\{0, 36, 64, 80, 100, 116\}$ is in $T_2$.
\end{itemize}

\vskip 0.5 truecm

The partition $\{\{3, 4, 8\}, \{2, 9, 16\}\}$ is the most
interesting. There are sets $T_1$ and $S_1$ containing
respectively 6 and 156 non-periodic tiles of cardinality 12 (type
$A$), corresponding to the cyclotomic factors $\Phi_{\{72, 48, 36,
24, 12, 8, 6, 4, 3\}}$ and $\Phi_{\{72, 48, 24, 12, 8, 6, 4,
3\}}$, respectively. Also, there are sets $T_2$ and $S_2$
containing respectively 12 and 48 non-periodic tiles of
cardinality 12 (type $B$), corresponding to the cyclotomic factors
$\Phi_{\{144, 36, 18, 16, 9, 2\}}$ and $\Phi_{\{144, 18, 16, 9,
2\}}$, respectively. Sets in $T_1$ tile with all sets in both
$T_2$ and $S_2$. Sets in $S_1$ only tile with sets in $T_2$. We
give a representative for each of $T_1, S_1, T_2, S_2$:
\begin{itemize}
\item
$\{0, 18, 28, 44, 54, 64, 80, 82, 98, 108, 118, 134\}$
is in $T_1$,
\item
$\{0, 16, 30, 44, 58, 74, 80, 94, 108, 110, 124, 138\}$
is in $S_1$,
\item
$\{0, 33, 40, 45, 48, 57, 88, 96, 105, 117, 129, 136\}$ is in $T_2$, and
\item
$\{0, 27, 30, 35, 60, 72, 75, 83, 102, 123, 131, 132\}$ is in $S_2$.
\end{itemize}

\vskip 0.5 truecm

The other partitions listed in Step 2 do not produce non-periodic
tilings.

\end{document}